\documentclass[8pt]{article}
\usepackage{indentfirst,latexsym,bm}
\usepackage{times}
\usepackage{amsfonts}
\usepackage{amssymb}
\usepackage[leqno]{amsmath}
\usepackage{dsfont}
\usepackage{amsthm}
\usepackage[all]{xy}
\usepackage{hyperref}
\usepackage{color,xcolor}
\begin{document}
	\title{On the Casimir number and formal codegree of Haagerup-Izumi fusion rings}
	\author{ Ying Zheng,  Jiacheng Bao and Zhiqiang Yu }
	\date{}
	\maketitle
	
	\newtheorem{theo}{Theorem}[section]
	\newtheorem{prop}[theo]{Proposition}
	\newtheorem{lemm}[theo]{Lemma}
	\newtheorem{coro}[theo]{Corollary}
	\theoremstyle{definition}
	\newtheorem{defi}[theo]{Definition}
	\newtheorem{exam}[theo]{Example}
	\newtheorem{conj}[theo]{Conjecture}
	\newtheorem{rema}[theo]{Remark}
	\newtheorem{ques}[theo]{Question}
	
	\newcommand{\A }{\mathcal{A}}
	\newcommand{\C }{\mathcal{C}}
	\newcommand{\FPdim}{\text{FPdim}}
	\newcommand{\FQ }{\mathbb{Q}}
	\newcommand{\FC }{\mathbb{C}}
	\newcommand{\Gal}{\text{Gal}}
	\newcommand{\HI}{\mathcal{HI}}
	\newcommand{\I }{\mathcal{I}}
	\newcommand{\Q }{\mathcal{O}}
	\newcommand{\Rep}{\text{Rep}}
	\newcommand{\unit}{\mathbf{1}}
	\newcommand{\vvec}{\text{Vec}}
	\newcommand{\Y }{\mathcal{Z}}
	\newcommand{\Z }{\mathbb{Z}}
	
	\abstract
	For any cyclic group $\Z_n$, we first determine the Casimir number  and determinant of the  Haagerup-Izumi fusion ring  $\HI_{\Z_n}$, it turns out that they do not share the same set of prime factors. Then we show that all finite-dimensional irreducible representations of $\HI_{\Z_n}$ are defined over certain cyclotomic fields. As a direct result, we obtain  the formal codegrees of  $\HI_{\Z_n}$, which  satisfy the pseudo-unitary inequality.

	\bigskip
	\noindent {\bf Keywords:} Haagerup-Izumi fusion ring; Casimir number; formal codegree
	
	Mathematics Subject Classification 2020: 18M20
	
	\section{Introduction}

The discovery of the Haagerup fusion category is  one of the results  in the classification of subfactors of small index \cite{AH,Ha,Iz}. Explicitly, the rank (i.e., the number of isomorphism classes of simple objects) of the  Haagerup fusion category  is $6$, and its  fusion rule  is as follows:
\begin{align*}
g\otimes X=X\otimes g^2, g^3=\unit, X\otimes X=\unit\oplus X\oplus gX\oplus g^2X.
\end{align*}
Surprisingly, currently we still do not know whether the Drinfeld center of the Haagerup fusion category can be obtained from the semisimplification of tensor categories of representations of quantum groups at root of unity. Meanwhile, with the help of the Cuntz algebra,  M. Izumi generalized the Haagerup fusion rules to a family of fusion rings determined by an arbitrary finite abelian group \cite{Iz}, which is called the Haagerup-Izumi fusion rings now.

Generally, let $G$ be a finite group. The Haagerup-Izumi fusion ring $\HI_G$ is a fusion ring with $\Z_+$-basis $\{g,gX|g\in G\}$ and they satisfy the following relations
	\begin{align*}
		g(hX) = ghX = (hX) g^{-1},\quad (gX) (hX) = gh^{-1}+\sum_{l\in G}lX,\quad \forall g,h \in G.
	\end{align*}
 M. Izumi characterized the existence of unitary fusion categories with Haagerup-Izumi fusion rules  using a family of nonlinear equations \cite{Iz}, and solutions were found for some groups of odd small orders. For example, if $G$ is trivial, then the corresponding fusion ring is Yang-Lee fusion ring; if $G=\Z_3$, then we get the Haagerup fusion ring (or category). As far as we know, by solving Izumi's equations,  there exist fusion categories with Haagerup-Izumi fusion rules for cyclic groups $\Z_n$ of  odd integers  less than or equal to $29$, see \cite{Bud,Iz} and the references therein.

On the one hand, in \cite{O1}, Ostrik defined the formal codegrees of fusion categories (or, generally fusion rings), which must be cyclotomic $d$-numbers (see \cite[Definition]{O1} for details). It was further proved that the formal codegrees of fusion categories satisfy the pseudo-unitary equation (Equation \ref{pseudounitary}), and they are useful in determining the structure of Drinfeld centers of fusion categories, see \cite{O2}. However, these properties fail for that of fusion rings, in general. Hence, we can use the formal codegrees to detect whether there is a fusion categories with  specific fusion rules.

On the other hand, the Casimir operator for Frobenius algebras was introduced in \cite{Lo}, and the author studied the Casimir operator for both  $H$ and the Grothendieck ring of  $\Rep(H)$, where $H$ is a Hopf algebra and $\Rep(H)$ is the category of finite-dimensional representations of $H$.  In \cite{WLL}, the authors defined the Casimir number and determinant of  fusion rings, and they investigated relations between these concepts and related concepts, such as the global dimension and Frobenius-Schur exponent of fusion categories, see \cite{EGNO} for specific definition. It  was asserted  in \cite{WLL} that the Casimir number and the determinant of the Grothendieck ring of a fusion category share the same set of prime factors. By determining  the Casimir number and  determinant of Haagerup-Izumi fusion rings $\HI_{\Z_n}$, we find that their statement is incorrect, see Proposition \ref{casimir-odd} and Proposition \ref{determinant}. Indeed, in the proof of \cite[Proposition 2.2]{WLL}, the authors used the commutativity of fusion rules inattentively.

This paper is organized as follows. In Section \ref{preliminaries}, we recall some definitions of fusion categories and fusion rings, we refer the reader to \cite{EGNO}. In Section \ref{section3}, we calculate  the Casimir number and determinant  of the Haagerup-Izumi fusion rings $\HI_{\Z_n}$,  and  we determine  all the irreducible representations and formal codegrees of  $\HI_{\Z_n}$ specifically (Lemma \ref{irre-rep-1} and Proposition \ref{irre-rep-2}), which then shows that $\HI_{\Z_n}$ satisfies the    the pseudo-unitary inequality (Theorem \ref{formdegreeHI}).
	\section{Preliminary}\label{preliminaries}
A  $\FC$-linear abelian category $\C$ is called a \emph{fusion category} if $\C$ is a finite semisimple tensor category \cite{EGNO}. In the following, we use $\Q(\C)$ and $\otimes$ to denote the set of isomorphism classes of simple objects of $\C$ and the monoidal product on $\C$, respectively.

Let $R$ be a fusion ring with $\Z_+$-basis $\{x_1,\cdots,x_s\}$. It is well-known \cite[Proposition 3.3.6]{EGNO} that the Frobenius-Perron homomorphism $\FPdim(-):R\to\FC$ is the unique ring homomorphism such that $\FPdim(x_j)\geq1$ for all  $x_j$. Moreover,  $\FPdim(x_j)$ is an algebraic integer and is defined as the Frobenius-Perron dimension of  $x_j$. The sum
\begin{align*}
\FPdim(R):=\sum_{j=1}^s\FPdim(x_j)^2
\end{align*}
is called  the Frobenius-Perron dimension of fusion ring $R$. For example, let $R=\HI_{\Z_n}$, then  $\FPdim(gX)=\frac{n+\sqrt{n^2+4}}{2}$, $g\in \Z_n$, and   $\FPdim(\HI_{\Z_n})=n\sqrt{n^2+4}(\frac{n+\sqrt{n^2+4}}{2})$.

Let $\C$ be a pivotal fusion category with a pivotal structure $j$, then  the  categorical dimensions $\dim(-)$ determined by $j$ induces a homomorphism from $\text{Gr}(\C)$ to $\FC$.  If   $\dim_j(X)=\dim_j(X^*)$ is true for all objects $X$ of $\C$, then $\C$ said to be  spherical \cite{EGNO}. And  the global dimension of $\C$ is defined as
\begin{align*}
\dim(\C):=\sum_{X\in\Q(\C)}\dim(X)^2.
\end{align*}

	Let $\C$ be a fusion category. The \emph{Casimir operator} of the Grothendieck ring $\text{Gr}(\C)$ (or, generally a fusion ring) \cite[Section 3.1]{Lo} is the map $c$ from $\text{Gr}(\C)$ to its center $Z(\text{Gr}(\C))$ given by
	\begin{align*}
		c(a)=\sum_{X\in\Q(\C)}XaX^*,~ \text{for}~a \in \text{Gr}(\C).
	\end{align*}
	The element $c(\unit)=\sum_{X\in\Q(\C)}XX^*$ is called the \emph{Casimir element} of $\text{Gr}(\C)$, where $\unit$ is the unit object of $\C$.   The intersection $\Z\cap\text{Im}(c)$ is a nonzero principal ideal of $\Z$, and the positive generator of $\Z\cap\text{Im}(c)$  is called the \emph{Casimir number} of $\C$ \cite{WLL}.
	It is well-known  that the matrix $[c(\unit)]$ of left multiplication by $c(\unit)$ with respect to the basis   of $\text{Gr}(\C)$ is a positive definite integer matrix, and  its determinant $\text{det}[c(\unit)]$ is called  the \emph{ determinant} of $\C$ \cite{WLL}.

	Given   a fusion category $\C$, let $\text{Irr}(\text{Gr}(\C))$ be the set of equivalence classes of irreducible representations of  the associative algebra $\text{Gr}(\C)\otimes_\Z\FC$. For   irreducible representations $\varphi,\varphi'\in \text{Irr}(\text{Gr}(\C))$, let $\text{Tr}_\varphi(-)$ be the ordinary trace function on the representation $\varphi$. Then there exists a central  element
	\begin{align*}
		\alpha_\varphi:=\sum_{X\in\Q(\C)}\text{Tr}_\varphi(X)X^*\in \text{Gr}(\C)\otimes_\Z\FC
	\end{align*}
	such that $\varphi'(\alpha_\varphi)=0$ if $\varphi\ncong\varphi'$, and $\alpha_\varphi$ acts on $\varphi$ as a   positive algebraic integer $f_\varphi$ \cite{Lu,O1}, called the  \emph{formal codegree }of $\C$ \cite{O1}.

It was shown in \cite[Lemma 2.6]{O1} that $c(\unit)$ acts on $\varphi$ as $f_\varphi \varphi(\unit)\cdot\text{id}$. For example, given a spherical fusion category $\C$,  the formal codegree determined by the Frobenius-Perron homomorphism and dimension homomorphism $\dim(-)$ are  the Frobenius-Perron dimension $\FPdim(\C)$ and global dimension $\dim(\C)$ of $\C$, respectively. Moreover, the formal codegrees of a spherical fusion category $\C$ satisfy the \emph{pseudo-unitary  equation }\cite[Theorem 2.21]{O2}
	\begin{align}\label{pseudounitary}
		\sum_{\varphi\in \text{Irr}(\text{Gr}(\C))}\frac{1}{f_\varphi^2}\leq\frac{1}{2}\left(1+\frac{1}{\dim(\C)}\right).
	\end{align}

	\section{The Casimir number and formal codegrees of $\mathcal{HI}_{\mathbb{Z}_n}$}\label{section3}
	In this section, we will first compute the Casimir number and determinant of $\HI_{\Z_n}$, then we give a complete equivalence classes of irreducible representations and the the corresponding formal codegrees of $\HI_{\Z_n}$.
	\subsection{The Casimir number and determinant}
	In the following, we denote
	\begin{align*}
		Y_0 = \unit, Y_1 = g,... \quad Y_{n-1} = g^{n-1}, Y_n = X, Y_{n+1} = gX, ... Y_{2n-1} = g^{n-1}X.
	\end{align*}

	\begin{theo}\label{casimir-odd}
		If n is odd, then the Casimir number of $\HI_{\Z_n}$ is $n(n^2+4)$;
	\end{theo}
	\begin{proof}
		Let $x= \sum_{i=0}^{2n-1}\lambda_{i}Y_{i} \in \HI_{\Z_n}$, where $\lambda_i\in\Z$.
		Then we have
		\begin{align*}
			c(x) &= \sum_{i=0}^{2n-1} Y_i x Y_i^* \\
			&= \big[2n\lambda_0 + n(\lambda_n + \cdots + \lambda_{2n-1})\big]\unit \\
			&\quad + \big[n\lambda_1 + n\lambda_{n-1} + n(\lambda_n + \cdots + \lambda_{2n-1})\big]Y_1 \\
			&\quad + \cdots \\
			&\quad + \big[n\lambda_{n-1} + n\lambda_1 + n(\lambda_n + \cdots + \lambda_{2n-1})\big]Y_{n-1} \\
			&\quad + \big[n(\lambda_0 + \lambda_1 + \cdots + \lambda_{n-1}) + (n^2 + 2)(\lambda_n + \cdots + \lambda_{2n-1})\big]Y_n \\
			&\quad + \cdots \\
			&\quad + \big[n(\lambda_0 + \lambda_1 + \cdots + \lambda_{n-1}) + (n^2 + 2)(\lambda_n + \cdots + \lambda_{2n-1})\big]Y_{2n-1}.
		\end{align*}
		It follows from the definition of the Casimir number that
		\begin{align*}
			n\lambda_{1}+n\lambda_{n-1}&+n(\lambda_{n}+...+\lambda_{2n-1}) = 0, \\
			n\lambda_{2}+n\lambda_{n-2}&+n(\lambda_{n}+...+\lambda_{2n-1}) = 0, \\
			&\vdots\\
			n\lambda_{n-1}+n\lambda_{1}&+n(\lambda_{n}+...+\lambda_{2n-1}) = 0, \\
			n(\lambda_{0}+\lambda_{1}+\cdots&+\lambda_{n-1})+(n^2+2)(\lambda_{n}+...+\lambda_{2n-1}) = 0.
		\end{align*}
Then we have
		\begin{align*}
			&2(\lambda_{1}+...+\lambda_{n-1})+(n-1)(\lambda_{n}+...+\lambda_{2n-1}) = 0,\\
			&2n\lambda_{0}+(n^2+n+4)(\lambda_{n}+...+\lambda_{2n-1}) = 0.
		\end{align*}
		Therefore,  $c(x) = [2n\lambda_{0}+n(\lambda_{n}+...+\lambda_{2n-1})]\unit = \frac{2n(n^2+4)}{n^2+n+4}\lambda_{0}\unit$. Notice that both   $n$ and   $(n^2+4)$ are relatively prime to $n^2+n+4$, and $(n^2+n+4)$ is even, so $\lambda_0=\frac{n^2+n+4}{2}$ is the smallest positive integer such that the ratio $\frac{2n(n^2+4)}{n^2+n+4}\lambda_{0}$ is an integer. Then we have   $c(x) = n(n^2+4)\unit$, and the Casimir number is $n(n^2+4)$ by definition.
	\end{proof}

	\begin{prop}\label{casimir-even}
		If $n$ is even, then the Casimir number  of the fusion ring $\HI_{\Z_n}$ is $\frac{n(n^2+4)}{2}$ if $4\mid n$, otherwise it is $n(n^2+4)$.
	\end{prop}
	\begin{proof}
		For any $ x= \sum\limits_{i=0}^{2n-1} \lambda_{i}Y_{i} \in \HI_{\Z_n}$, where $\lambda_i\in\Z$, similarly,
		we have:
		\begin{align*}
			c(x) &= \sum_{i=0}^{2n-1} Y_i x Y_i^* \\
			&= \big[2n\lambda_0 + n(\lambda_n + \cdots + \lambda_{2n-1})\big]\unit \\
			&\quad + \big[n\lambda_1 + n\lambda_{n-1} + n(\lambda_{n} + \cdots + \lambda_{2n-1})\big]Y_1 \\
			&\quad + \cdots \\
			&\quad + \big[n\lambda_{n-1} + n\lambda_1 + n(\lambda_{n} + \cdots + \lambda_{2n-1})\big]Y_{n-1} \\
			&\quad + \big[n(\lambda_0 + \lambda_1 + \cdots + \lambda_{n-1}) + 4(\lambda_n + \lambda_{n+2} + \cdots + \lambda_{2n-2}) + n^2(\lambda_{n}+...+\lambda_{2n-1})\big]Y_{n} \\
			&\quad + \big[n(\lambda_0 + \lambda_1 + \cdots + \lambda_{n-1}) + 4(\lambda_{n+1} + \lambda_{n+3} + \cdots + \lambda_{2n-1}) + n^2(\lambda_{n}+...+\lambda_{2n-1})\big]Y_{n+1} \\
			&\quad + \cdots \\
			&\quad + \big[n(\lambda_0 + \lambda_1 + \cdots + \lambda_{n-1}) + 4(\lambda_n + \lambda_{n+2} + \cdots + \lambda_{2n-2}) + n^2(\lambda_{n}+...+\lambda_{2n-1})\big]Y_{2n-2} \\
			&\quad + \big[n(\lambda_0 + \lambda_1 + \cdots + \lambda_{n-1}) + 4(\lambda_{n+1} + \lambda_{n+3} + \cdots + \lambda_{2n-1}) + n^2(\lambda_{n}+...+\lambda_{2n-1})\big]Y_{2n-1}
		\end{align*}
		From the definition of Casimir number, 	
		\begin{align*}
			n\lambda_{1}+n\lambda_{n-1}&+n(\lambda_{n}+...+\lambda_{2n-1}) = 0,\\
			n\lambda_{2}+n\lambda_{n-2}&+n(\lambda_{n}+...+\lambda_{2n-1}) = 0,\\
			&\vdots\\
			n\lambda_{n-1}+n\lambda_{1}&+n(\lambda_{n}+...+\lambda_{2n-1}) = 0,\\
			n(\lambda_{0}+...+\lambda_{n-1})&+4(\lambda_{n}+\lambda_{n+2}+...+\lambda_{2n-2})+n^2(\lambda_{n}+...+\lambda_{2n-1}) = 0,\\
			n(\lambda_{0}+...+\lambda_{n-1})&+4(\lambda_{n+1}+\lambda_{n+3}+...+\lambda_{2n-1})+n^2(\lambda_{n}+...+\lambda_{2n-1})= 0.
		\end{align*}
These equations imply
		\begin{align*}
			&2(\lambda_{1}+...+\lambda_{n-1})+(n-1)(\lambda_{n}+...+\lambda_{2n-1}) = 0,\\
			&n(\lambda_{0}+\lambda_{1}+...+\lambda_{n-1})+(n^2+2)(\lambda_{n}+...+\lambda_{2n-1}) = 0,\\
			&2n\lambda_{0}+(n^2+n+4)(\lambda_{n}+...+\lambda_{2n-1}) = 0.
		\end{align*}
Therefore,  $\lambda_{n}+...+\lambda_{2n-1} =- \frac{2n}{n^2+n+4}\lambda_0$, it follows that $c(x) = \frac{2n(n^2+4)}{n^2+n+4}\lambda_0\unit$. Since $n$ is even, the largest common divisor  $(n^2+n+4,n)=\left\{  \begin{array}{ll}
            4, & \hbox{if $4\mid n$} \\
  2, & \hbox{if $4\nmid n$.}
        \end{array}
\right.$. By using a similar argument as  Proposition \ref{casimir-odd},    the Casimir number of Haagerup-Izumi fusion ring $\HI_{\Z_n}$ is $\frac{n(n^2+4)}{2}$ if $4\mid n$, otherwise, it is $(n^2+4)n$.\end{proof}
	
	Next we compute the determinant of the Haagerup-Izumi fusion ring $\HI_{\Z_n}$.	
	\begin{prop}\label{determinant}
		The determinant of $\mathcal{HI}_{\mathbb{Z}_n}$ is $n^{2n}2^{2n-2}(n^2+4)$.
	\end{prop}
	\begin{proof}
		By definition, we have  $c(\unit)=2n\unit+n(Y_n+Y_{n+1}+...+Y_{2n-1})$. Then we have
		\begin{align*}
			c(\unit)(Y_0,Y_1,Y_2,...,Y_{2n-1})=(Y_0,Y_1,Y_2,...,Y_{2n-1})\left({\begin{array}{cc}
					2n\cdot I_{n} & B\\ B & D
			\end{array}}\right),
		\end{align*}	
		where $B=\left(
		\begin{array}{cccc}
			n & n&\cdots & n \\
			n&n&\cdots & n\\
			\vdots &\vdots&\ddots&\vdots\\
			n & n&\cdots & n \\
		\end{array}
		\right)
		$ and $D=\left(
		\begin{array}{cccc}
			n^2+2n&n^2&\cdots&n^2  \\
			n^2& n^2+2n & \cdots & n^2 \\
			\vdots & \vdots & \ddots & \vdots \\
			n^2 & n^2 & \cdots & n^2+2n \\
		\end{array}
		\right)$.
		Therefore, we have \begin{align*}
			\det\bigg(\left({\begin{array}{cc}
					2n\cdot I_{n} & B\\ B & D
			\end{array}}\right)\bigg)=\det\bigg(\left({\begin{array}{cc}
					2n\cdot I_{n} & 0\\ B & D-\frac{1}{2}B^2
			\end{array}}\right)\bigg)=n^{2n}2^{2n-2}(n^2+4).
		\end{align*}
		This finishes the proof of the proposition.
	\end{proof}
	\begin{rema}If  $ n$ is odd, Theorem \ref{casimir-odd} states that the Casimir number of $\mathcal{HI}_{\Z_n}$ is $n(n^2+4)$, which is odd too, so it does not share  the same set of prime divisors with the determinant of $\mathcal{HI}_{\Z_n}$ by Proposition \ref{determinant}. Note that there exist fusion categories with Haagerup-Izumi fusion rules for odd integers less than $31$ \cite{Bud}, therefore, this provides a counterexample to \cite[Corollary 2.4]{WLL}. Indeed,  in the proof of \cite[Proposition 2.2]{WLL}, the authors  used the commutativity of the fusion rules of fusion categories implicitly.
	\end{rema}

	\subsection{Irreducible representations and formal codegrees of $\mathcal{HI}_{\mathbb{Z}_n}$}
	In this subsection, we  determine all the irreducible representations  of the associative algebra $A:=\HI_{\Z_n}\otimes_\Z\FC$ and  the  formal codegrees of $\HI_{\Z_n}$.
	
	Note the fusion rules of  $\mathcal{HI}_{\mathbb{Z}_n}$ is determined by $Y_1$ and $Y_n$ with relations
	\begin{align}\label{genEquat}
		Y_1^n=1, Y_1Y_n=Y_nY_1^{n-1}, Y_n^2=\unit+\sum_{i=1}^nY_1^iY_n.
	\end{align}
	Let $\rho$ be an one-dimensional representation of $A$. Denote $x:=\rho(Y_1)$ and $y:=\rho(Y_n)$.
	
	\begin{lemm}\label{irre-rep-1}
		If $n$ is odd, then $\rho$ is given by
		\begin{align*}
			\left\{
			\begin{array}{ll}
				x=1,   \\
				y=\frac{n+\sqrt{n^2+4}}{2},
			\end{array}
			\right.\left\{
			\begin{array}{ll}
				x=1,   \\
				y=\frac{n-\sqrt{n^2+4}}{2};
			\end{array}
			\right.
		\end{align*}
		if $n$ is even, then $\rho$ is isomorphic to one of the following
		\begin{align*}
			\left\{
			\begin{array}{ll}
				x=1,   \\
				y=\frac{n+\sqrt{n^2+4}}{2},
			\end{array}
			\right.\left\{
			\begin{array}{ll}
				x=1,   \\
				y=\frac{n-\sqrt{n^2+4}}{2},
			\end{array}
			\right.
			\left\{
			\begin{array}{ll}
				x=-1,   \\
				y=1,
			\end{array}
			\right.\left\{
			\begin{array}{ll}
				x=-1,   \\
				y=-1.
			\end{array}
			\right.
		\end{align*}
	\end{lemm}
	\begin{proof}
		The Equations (\ref{genEquat}) imply that $x^n=1$ and $xy=x^{n-1}y$. Note that $y\neq0$, so $x$ is also a $(n-2)$-th root of unity. Therefore,  $x=1$ if $n$ is odd, then $y=\frac{n\pm\sqrt{n^2+4}}{2}$. If $n$ is even and $x\neq1$, then   $x=-1$ and $y^2=1$ by Equations (\ref{genEquat}) .
	\end{proof}
If $n=1,2$, then $\HI_{\Z_n}$ is commutative, it is easy to see that homomorphisms in Lemma \ref{irre-rep-1} are all the isomorphism classes of irreducible representations of $A$. We assume $n\geq3$ below.
	
	Let $V_i$ be a two-dimensional vector space  over $\FC$, and let $\{u_1^{(i)},u_2^{(i)}\}$ be a basis of $V_i$,  where $1\leq i\leq \big[\frac{n-1}{2}\big]$, and $[a]$ denotes the  maximal integer that is less than or equal to $a$. We define
	\begin{align}\label{modrelation}
		\left\{
		\begin{array}{ll}
			Y_1\cdot u_1^{(i)}:=\zeta_n^iu_1^{(i)},~ Y_1\cdot u_2^{(i)}:=\zeta_n^{-i} u_2^{(i)},   \\
			Y_n\cdot u_1^{(i)}:=u_2^{(i)}, \quad Y_n\cdot u_2^{(i)}:=u_1^{(i)},
		\end{array}
		\right.
	\end{align}
	where $\zeta_n$ is a $n$-th primitive root of unity. Then we extend this action $\FC$-linearly to $A$.

\begin{lemm}
The Equations (\ref{modrelation}) define an $A$-module structure on all  $V_i$.
\end{lemm}
\begin{proof}
Let $M_i:=\left({\begin{array}{cc}
					\zeta_n^i & 0\\ 0 & \zeta_n^{-i}
			\end{array}}\right)$ and $N_i:=\left({\begin{array}{cc}
					0& 1\\ 1 & 0
			\end{array}}\right)$,  where $1\leq i\leq \big[\frac{n-1}{2}\big]$. A direct computation shows that these two matrices satisfies the generating the Equations \ref{genEquat}, so the $A$-module structure is well-defined.
\end{proof}
	\begin{prop}\label{irre-rep-2}
		Let $V$ be an irreducible representation of $A$ with $\dim_\FC(V)>1$. Then $V$ is isomorphic to one of the $V_i$ defined by the Equations (\ref{modrelation}), $1\leq i\leq \big[\frac{n-1}{2}\big]$.
	\end{prop}
	\begin{proof}
		For any  $1\leq i\leq\big[\frac{n-1}{2}\big]$, by Equations (\ref{modrelation}), we have
		\begin{align*}
			Y_1(u_1^{(i)},u_2^{(i)})=(u_1^{(i)},u_2^{(i)})\left({\begin{array}{cc}
					\zeta_n^i & 0\\ 0 & \zeta_n^{-i}
			\end{array}}\right)=(u_1^{(i)},u_2^{(i)})M_i.
		\end{align*}
		Note that the traces $\text{tr}(M_i)=\zeta_n^i+\zeta_n^{-i}$ are different from each other for $1\le i \le [\frac{n-1}{2}]$, which means   these  representations $V_i$ of $A$ are not isomorphic to each other.

Next, we show that each $V_i$ is   irreducible. In fact, assume  $V_i$ is reducible for some $i$, and let $V'$ be a proper sub-representation of $V_i$, obviously $\dim_\FC(V')=1$. One the one hand, it follows from  the Equations (\ref{modrelation}) that $V'\neq \FC u_1^{(i)}$ and $V'\neq \FC u_2^{(i)}$. Therefore,   $V'=\FC(u_1^{(i)}+\lambda u_2^{(i)})$ for some nonzero scalar  $\lambda$, so  $u_1^{(i)}+\lambda u_2^{(i)}$ must be an eigenvector of $Y_1$. On the other hand, by definition, we have
		\begin{align*}
			Y_1(u_1^{(i)}+\lambda u_2^{(i)})=\zeta_n^iu_1^{(i)}+\lambda \zeta_n^{-i}u_2^{(i)}
			=\zeta_n^i(u_1^{(i)}+\lambda \zeta_n^{-2i} u_2^{(i)}).
		\end{align*}
		Hence, we have  $\zeta_n^{2i}=1$.  However, $\zeta_n$ is a primitive $n$-th root of unity and $2i<n$ for all $1\leq i\leq [\frac{n-1}{2}]$, it is impossible. So $V_i$ is irreducible for all $i$.
		
		Notice that
		\begin{align*}
			\dim_\FC(A)=2n=\left\{
			\begin{array}{ll}
				2+\sum_{i=1}^{\frac{n-1}{2}}\dim_\FC(V_i)^2, & \hbox{if $n$ is odd,} \\
				4+\sum_{i=1}^{\frac{n-2}{2}}\dim_\FC(V_i)^2, & \hbox{if $n$ is even,}
			\end{array}
			\right.
		\end{align*}
		hence, together with Lemma \ref{irre-rep-1}, the Wedderburn' s Theorem shows that any irreducible representation with dimension being
 larger than $1$ must be isomorphic to one of these $V_i$  ($1\leq i\leq [\frac{n-1}{2}])$. This finishes the proof of the proposition.
	\end{proof}
We note that  irreducible  representations for fusion rings with two orbits are also constructed in \cite[Proposition 2.16]{Sch}.

Recall that the Drinfeld center $\Y(\C)$ of a   fusion category $\C$ is a  non-degenerate braided fusion category, and there is a surjective tensor functor $F$ (called the forgetful functor) from $\Y(\C)$ to $\C$.
	\begin{theo}\label{formdegreeHI}
If $n$ is odd, then the formal codegrees of $\HI_{\Z_n}$ are
\begin{align*}
\FPdim(\HI_{\Z_n}), \sigma(\FPdim(\HI_{\Z_n})),   \overbrace{n,\cdots,n}^{\frac{n-1}{2}},
\end{align*}
 and if $n$ is even, then the formal codegrees of the fusion ring $R$ are
\begin{align*}
\FPdim(\HI_{\Z_n}), \sigma(\FPdim(\HI_{\Z_n})), 2n,2n,  \overbrace{n,\cdots,n}^{\frac{n-2}{2}},
\end{align*}
where $\sigma$ is a generator of $\Gal(\FQ(\sqrt{n^2+4})/\FQ)$. Therefore, the formal codegrees of the fusion ring $\HI_{\Z_n}$ satisfy the pseudo-unitary inequality.
	\end{theo}
\begin{proof}
Let $n\geq3$ be an odd integer, the other case is similar.  A direct computation shows  \begin{align*}
c(\unit)=2n\unit+n\sum_{j=n}^{2n-1}Y_j.
\end{align*}
 Then Lemma \ref{irre-rep-1}  shows that $c(\unit)$ acts on one-dimensional representations as scalar $\FPdim(\HI_{\Z_n})$ and $\sigma(\FPdim(\HI_{\Z_n}))$, which are formal codegrees of $\HI_{\Z_n}$. And $c(\unit)$ acts on all two-dimensional irreducible representations as $2n$ by Proposition \ref{irre-rep-2}, so all the corresponding  formal codegrees of $\HI_{\Z_n}$ are  $n$ by definition. Meanwhile,
\begin{align*}
\sum_{\chi}\frac{1}{f_\chi^2}&=\frac{n-1}{2n^2}+\frac{1}{\sigma(\FPdim(\HI_{\Z_n}))^2}
+\frac{1}{\FPdim(\HI_{\Z_n})^2}\\
&=\frac{n^2+n+4}{2n(n^2+4)}<\frac{1}{2}\left(1+\frac{1}{\dim(\C)}\right).
\end{align*}
This finishes the proof of the theorem.
\end{proof}
\begin{lemm}\label{dimenHI}
Let $\C$ be a fusion category $\C$ with $\text{Gr}(\C)=\HI_{\Z_n}$. Then $\dim(\C)\neq n,2n$. In particular, $\dim(\C)$ is a Galois conjugate of $\FPdim(\HI_{\Z_n})$.
\end{lemm}
\begin{proof}
Notice that the rank of $\C$ is  $2n$, \cite[Theorem 4.2.2]{O3} states that there is a Galois conjugate of $\dim(\C)$ that is no less than $2n$, so $\dim(\C)\neq n$. Moreover, $\dim(\C)=2n$ if and only if all simple objects of $\C$ has Frobenius-Perron dimension $1$
by \cite[Remark 4.2.3]{O3}, which is impossible for Haagerup-Izumi fusion rings.
\end{proof}

The following corollary is a direct result of \cite[Theorem 2.13]{O2}, Theorem \ref{formdegreeHI} and Lemma \ref{dimenHI}.
\begin{coro}
Let $\C$ be a spherical fusion category with Haagerup-Izumi fusion rules determined by $\Z_n$, where $n\geq3$. Let $I:\C\to\Y(\C)$ be the adjoint functor to forgetful functor $F:\Y(\C)\to\C$. Then
\begin{align*}
I(\unit)=
\left\{
  \begin{array}{ll}
    \unit\oplus X_1\oplus 2Z_1\oplus\cdots\oplus2Z_{\frac{n-1}{2}}, & \hbox{if $n$ is odd;} \\
    \unit\oplus X_1\oplus W_1\oplus W_2\oplus 2Z_1\oplus\cdots\oplus2Z_{\frac{n-2}{2}}, & \hbox{if $n$ is even.}
  \end{array}
\right.
\end{align*}
where $\dim(X_1)=\frac{\dim(\C)}{\sigma(\dim(\C))}$, $\dim(W_1)=\dim(W_2)=\frac{\dim(\C)}{2n}$, and $\dim(Z_j)=\frac{\dim(\C)}{n}$, $1\leq j\leq [\frac{n-1}{2}]$.
\end{coro}	

	\section*{Acknowledgements}
	The authors thanks L. Li and A. Schopieray for comments. The third author was supported by NSFC (no. 12571041) and  Qinglan Project of Yangzhou University.
	
	\bigskip\author{{Ying Zheng, \thanks{Email:\,yzhengmath@yzu.edu.cn}, \\
			Jiacheng Bao, \thanks{Email:\,MX120240308@stu.yzu.edu.cn},\\
{Zhiqiang Yu, \thanks{Email:\,zhiqyumath@yzu.edu.cn}\\
{\small School  of Mathematics,  Yangzhou University, Yangzhou 225002, China}}

	\end{document}